\documentclass[12pt]{amsart}
\pdfoutput=1
\usepackage{amssymb}
\usepackage{amsmath}
\usepackage{amsfonts}
\usepackage[usenames]{color}
\usepackage{graphicx}
\usepackage{array}
\usepackage{psfrag}
\usepackage{color}
\usepackage{ulem}
\usepackage{fancyhdr}

\usepackage{caption}

\usepackage{import}
\usepackage{xifthen}
\usepackage{pdfpages}
\usepackage{transparent}

\makeatletter
 
 \@addtoreset{equation}{section}
\makeatother

\newtheorem{theorem}{Theorem}
\newtheorem{lemma}[theorem]{Lemma}
\newtheorem{proposition}[theorem]{Proposition}

\theoremstyle{definition}

\newtheorem{example}[theorem]{Example}

\newtheorem{corollary}[theorem]{Corollary}

\theoremstyle{remark}

\numberwithin{equation}{section}

\newcommand{\R}{\mathbb{R}}

\newcommand{\mleq}{\preccurlyeq}
\newcommand{\mgeq}{\succcurlyeq}

\begin{document}

\title{{Complete Minors in Complements of Non-Separating Planar Graphs}}

\date{\today}
\author{
Leonard Fowler, Gregory Li, and Andrei Pavelescu
}

\address{ \textit{fowlel@rpi.edu}, Rensselaer Polytechnic Institute, Troy, NY 12180}
\address{\textit{gregoryli@college.harvard.edu}, Harvard University,
Cambridge, MA 02138}
\address{\textit{andreipavelescu@southalabama.edu}, Department of Mathematics and Statistics, University of South Alabama, Mobile, AL  36688}

\maketitle
\rhead{Complete Minors in Complements of Non-Separating Planar Graphs}

\begin{abstract}

We prove that the complement of any non-separating planar graph of order $2n-3$ contains a $K_n$ minor, and argue that the order $2n-3$ is lowest possible with this property. 
To illustrate the necessity of the non-separating hypothesis, we give an example of a planar graph of order 11 whose complement does not contain a $K_7$ minor.  We argue that the complements of planar graphs of order 11 are intrinsically knotted. We compute the Hadwiger numbers of complements of wheel graphs.
\end{abstract}

\section{Introduction}

A planar graph $G$ is \textit{non-separating planar} if and only if there exists a \textit{planar embedding} of $G$ in $\mathbb{R}^2$ where for any cycle $C \subseteq G$, the vertices of $G\backslash C$ are not separated between the region $R$ in $\R^2$ enclosed by $C$ and the region $\R^2 \backslash R$.
 Dehkordi and Farr \cite{DF} classified maximal non-separating planar graphs as (1) wheel graphs, (2) elongated triangular prism graphs, or (3) maximal outerplanar graphs.
In \cite{PP}, Pavelescu and Pavelescu proved that the complements of non-separating planar graphs of order ten are intrinsically knotted. In this article, using the classification in  \cite{DF},  we prove that the complements of  non-separating planar graphs of order 11 contain $K_7$ as a minor, and we generalize this result by the following theorem:

\begin{theorem}
\label{mainIK}
Let $G$ be a non-separating planar graph of order $2n-3$, with $n\ge 7$. Then $\overline{G}$, the complement of $G$, admits a $K_n$ minor.
\end{theorem}

While this result does not hold for arbitrary planar graphs of order $2n-3$, we conjecture that the complements of planar graphs of order (at least) 11 are intrinsically knotted. 
An embedding of a graph in space is a \textit{knotted embedding} if there exists a cycle which forms a nontrivial knot. A graph $G$ is \textit{intrinsically knotted} (IK) if all embeddings of $G$ in $\mathbb{R}^3$ are knotted embeddings. The class of \textit{knotlessly embeddable} (nIK) graphs, consisting of graphs which are \textit{not} IK, is a minor-closed family of graphs. By the Robertson--Seymour theorem \cite{RS}, the class of nIK graphs possesses a forbidden minor characterization. While there are over 264 known minor minimal IK (MMIK) graphs \cite{GMN}, a complete list of forbidden minors is not yet available. 
Among those which are known are $K_7$, shown to be IK by the work of  Conway and Gordon \cite{CG}, and the complete 4-partite graph $K_{3, 3, 1, 1}$, shown to be MMIK by the work  of Foisy \cite{Foisy}. 
This means a graph  $G$ which contains  $K_7$ or $K_{1, 1, 3, 3}$ as a minor is IK. 
We investigate graphs of order 11 and show that the complement of non-separating planar graphs of order 11 must contain a $K_7$ minor.

\section{Notation and Definitions}

We denote the set of vertices and edges of a simple, undirected, and finite graph $G$ as $V(G)$ and $E(G)$, respectively. For each pair of vertices $u, v \in V(G)$ we say $(u, v) \in E(G)$ if and only if $u$ and $v$ are adjacent in $G$. If $(u, v) \in E(G)$, we say that the edge $(u,v)$ is \textit{incident} to both $u$ and $v$. The set of all vertices $u \in V(G)$ satisfying $(u, v) \in E(G)$ for a fixed vertex $v$ is the open neighbor set $N_G(v)$, or $N(v)$ when $G$ is clear from context. For each graph $G$, we denote its (edge) \textit{complement} graph $\overline{G} = (V(G),E^\neg(G))$, such that $$E^\neg(G) = \{ (u, v) | u, v \in V(G) \: \text{and} \: (u, v) \not \in E(G) \}.$$

\noindent 
If a graph  $G'$ is a subgraph of $G$ we write $G' \subseteq G$. 
We say a graph $H$ is a \textit{minor} of $G$, or $H \mleq G$, if there exists an isomorphism from a subdivision of $H$ to a subgraph $G' \subseteq G$. If a graph isomorphic to $H$ can be obtained by applying a (possibly empty) series of vertex deletions, edge deletions, and edge contractions to $G$, then also $H \mleq G$. 
The minor relation $\mleq$ is reflexive, anti-symmetric, and transitive. 
A graph $G$ is \textit{planar} when there exists an embedding of $G$ in $\mathbb{R}^2$ in which edges intersect only at points in $V(G)$.

We define the \textit{join} $H = H_1+H_2$ of two graphs $H_1$ and $H_2$ as follows. 
The set of vertices of $H$ is $V(H) = V(H_1) \sqcup V(H_2)$, the disjoint union of the sets of vertices of $H_1$ and $H_2$.
 In addition to existing adjacencies in $H_1$ and $H_2$, each vertex in $H$ corresponding to a vertex in $H_1$ is adjacent to every vertex arising from $H_2$, so that $E(H) = E(H_1) \sqcup E(H_2) \sqcup (V(H_1) \times V(H_2))$, where $V(H_1) \times V(H_2)$ is the Cartesian product of the vertex sets. For any $H' \cong H$, we say $H' = H_1 + H_2$.

\section{Maximal Non-Separating Planar Graphs}

We begin by noting that if $H$ is a minor of $G$ of the same order, $ \overline{G}$ is a subgraph of  $\overline{H}$. This allows us to only consider \textit{maximal} non-separating planar graphs of a given order.
We organize the proof of Theorem \ref{mainIK} following the classification of non-separating planar graphs by Dehkordi and Farr \cite{DF}.
\begin{theorem}[Dehkordi, Farr \cite{DF}]
A graph $G$ is maximal non-separating planar if and only if $G$ belongs to one of the following categories:
\begin{enumerate}
\item wheel graphs
\item elongated triangular prism graphs
\item maximal outerplanar graphs
\end{enumerate}
\label{numtheo}
\end{theorem}

\begin{figure}[htpb!]
\begin{center}
\begin{picture}(390, 100)
\put(0,0){\includegraphics[width=5.4in]{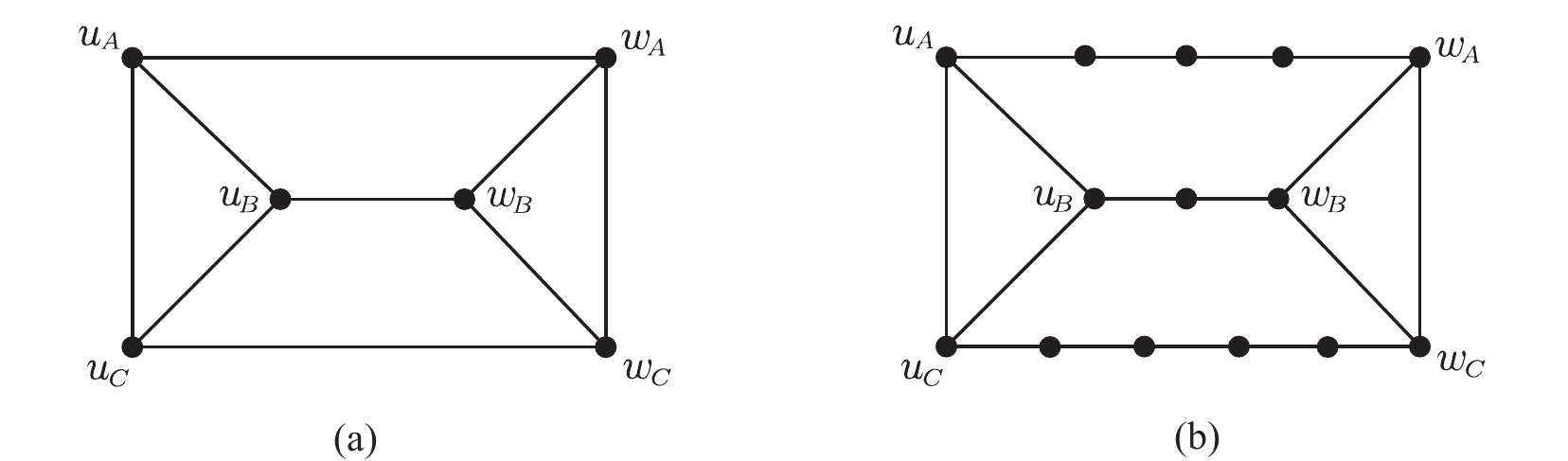}}
\end{picture}
\caption{\small (a) The triangular prism graph; (b) Elongated triangular prism.}
\label{prism}
\end{center}
\end{figure} 
\noindent For well-defined classes of graphs, the cardinality of the vertex set will be denoted in subscript. The \textit{wheel graph} $W_n$ on $n$ vertices (for $n\geq 4$) is isomorphic to the join of the cycle graph $C_{n-1}$ and a vertex; equivalently, $W_n = K_1 + C_{n-1}$. 
In what follows, we denote the vertices of $W_n$ by $v_1, v_2, \ldots , v_n$, where $v_1, v_2, \ldots , v_{n-1}$ represent the vertices of $C_{n-1}$ in clockwise order, and $v_n$ is adjacent to all $v_i$, $i=1,2,\ldots, n-1$.
An \textit{elongated triangular prism graph} is any graph constructed from the triangular prism graph  in Figure \ref{prism}(a) by consecutively subdividing any or all of the edges $(u_A,w_A)$, $(u_B,w_B)$, and $(u_C,w_C)$. An example is given in Figure \ref{prism}(b).
An outerplanar graph is a graph which has a planar embedding in which all vertices lie on a single face. 
A graph which is maximal with this property is \textit{maximal outerplanar}.
A maximal outerplanar graph can be represented by an $n-$cycle with its interior triangulated in the plane.
In the next three sections, we look at each type of maximal non-separating planar graphs.


\section{Wheel Graphs}

\begin{theorem}
The complement $\overline{W_{n}}$ of the wheel graph $W_{n}$ satisfies $K_{\lfloor{\frac{3(n-1)}{4}}\rfloor} \mleq \overline{W_{n}}$ for $n \geq 6$.
\label{theo:wheel}
\end{theorem}

\begin{proof}

We shall consider the remainder of $n$ modulo 4. 

For  $n=4t+1$, in $\overline{W_{n}}$, the vertices $v_1, v_3, ....,v_{4t-1}$ induce a complete subgraph on $2t$ vertices.
Contracting the edges $(v_{2l},v_{2l+2t})$ for $1\le l \le t$, produces a minor of $\overline{W_{n}}$ isomorphic to $K_{3t}$. See Figure \ref{wheelcomps}(c) for the case $n=13$, $t=3$.

\begin{figure}[htpb!]
\begin{center}
\begin{picture}(420, 150)
\put(-20,0){\includegraphics[width=6.5in]{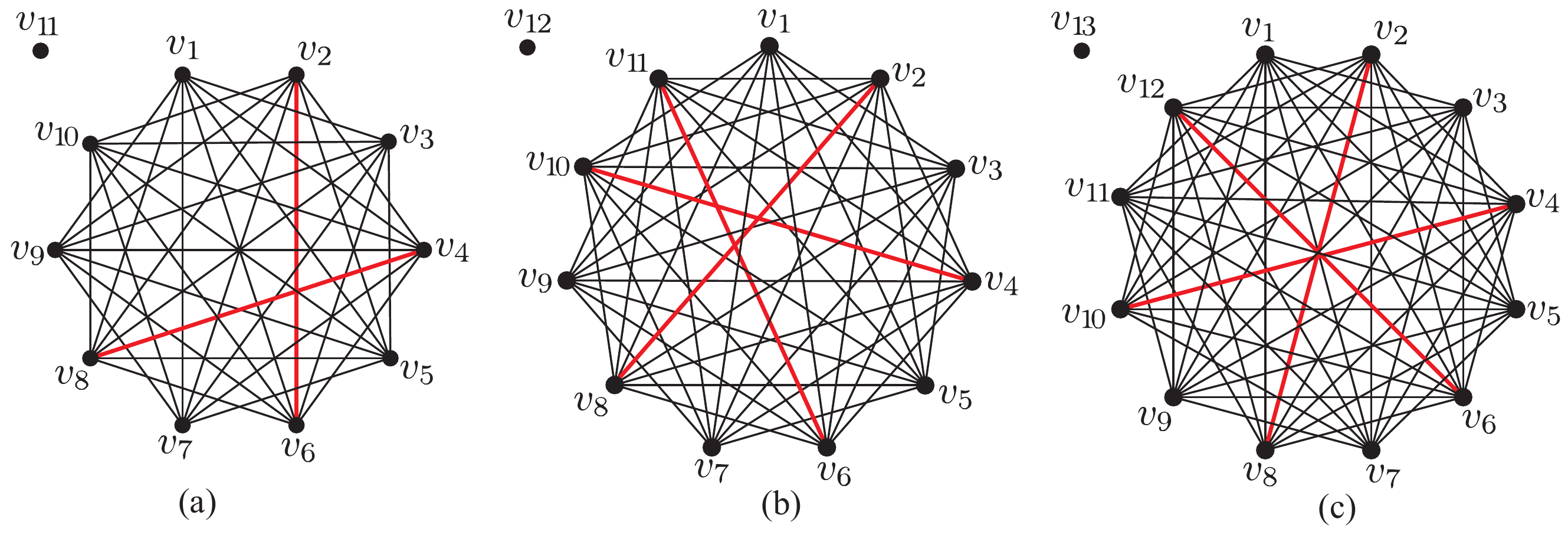}}
\end{picture}
\caption{\small (a) $\overline{W_{11}}$, contracting the highlighted edges gives a $K_7$ minor; (b) $\overline{W_{12}}$,  contracting the highlighted edges gives a $K_8$ minor; (b) $\overline{W_{13}}$,  contracting the highlighted edges gives a $K_9$ minor.}
\label{wheelcomps}
\end{center}
\end{figure} 

For $n=4t+2$, in $\overline{W_{n}}$, the vertices $v_1, v_3, ....,v_{4t-1}$ induce a complete subgraph. Contracting the edges $(v_{2l},v_{2l+2t})$ for $1\le l \le t$, and contracting any of the edges incident to $v_{4t+1}$, produces a minor of $\overline{W_{n}}$ isomorphic to $K_{3t}$. 

For $n=4t+3$, in $\overline{W_{n}}$, the vertices $v_1, v_3, ....,v_{4t+1}$ induce a complete subgraph of order $2t+1$. 
Contracting the edges $(v_{2l},v_{2l+2t})$ for $1\le l \le t$ and deleting the vertex $v_{4t+2}$  produces a minor of $\overline{W_{n}}$ isomorphic to $K_{3t+1}$. See Figure \ref{wheelcomps}(a) for the case $n=11$, $t=2$.

For $n=4t+4$, in $\overline{W_{n}}$, the vertices $v_1, v_3, ....,v_{4t+1}$ induce a complete subgraph of order $2t+1$. 
Contracting the edges $(v_{2l},v_{2l+2t+2})$ for $1\le l \le t$, and $(v_{2t+2},v_{4t+3})$, produces a minor of $\overline{W_{n}}$ isomorphic to $K_{3t+2}$. See Figure \ref{wheelcomps}(b) for the case $n=12$, $t=2$.
\end{proof}

%

\begin{corollary} 
\label{wheelgeneral1}
The complement $\overline{W_{2t-3}}$ of the wheel graph $W_{2t-3}$ satisfies $K_t \mleq \overline{W_{2t-3}}$ for $t \geq 6$.
\label{cor}
\end{corollary}
\begin{proof}
We apply Theorem \ref{theo:wheel} with $n = 2t -3$. 
We have as a result $$K_{\lfloor{\frac{3(2t-4)}{4}}\rfloor} =  K_{\lfloor{\frac{3t}{2}}\rfloor - 3} \mleq \overline{W_{2t-3}}.$$ 
For $t \geq 6$, we have $\lfloor{\frac{3t}{2}}\rfloor - 3 \geq t$ and thus $$K_t \mleq K_{\lfloor{\frac{3t}{2}}\rfloor - 3} \mleq \overline{W_{2t-3}}.$$
\end{proof}

The following theorem shows that the value $\lfloor{\frac{3(n-1)}{4}}\rfloor$ is the best possible.

\begin{theorem}
The edge complement $\overline{W_{n}}$ of the wheel graph $W_{n}$ has no minor isomorphic to $K_{\lfloor{\frac{3(n-1)}{4}}\rfloor + 1}$, equivalently $K_{\lfloor{\frac{3(n-1)}{4}}\rfloor + 1} \not \mleq \overline{W_{n}}$ for $n\geq 6$.
\label{sharpwheelK}
\end{theorem}
\begin{proof}
Assume on the contrary, $K_{\lfloor{\frac{3(n-1)}{4}}\rfloor + 1} \mleq \overline{W_{n}}$. 
The graph $\overline{W_{n}}$ contains an isolated vertex, $v_{n}$. Because $n\geq 1$, we have $K_{\lfloor{\frac{3(n-1)}{4}}\rfloor + 1} \mleq \overline{W_{n}} \iff K_{\lfloor{\frac{3(n-1)}{4}}\rfloor + 1} \mleq \overline{W_{n}}\backslash v_{n}$. 

As the minor is complete, it follows there is a sequence of $\lfloor{\frac{n-2}{4}}\rfloor$ edge contractions applied to $\overline{W_{n}}\backslash v_{n}$ which results in a graph isomorphic to $K_{\lfloor{\frac{3(n-1)}{4}}\rfloor + 1}$. 
Let us enumerate the intermediate graphs thus obtained by the sequence $(H_1, H_2, \dots, H_{\lfloor{\frac{n-2}{4}}\rfloor - 1})$ where $$\overline{W_{n}} \mgeq \overline{W_{n}}\backslash v_{n} \mgeq H_1 \mgeq H_2 \mgeq \cdots \mgeq H_{\lfloor{\frac{n-2}{4}}\rfloor - 1} \mgeq K_{\lfloor{\frac{3(n-1)}{4}}\rfloor + 1}.$$

In general, denoting $H_{-1} = \overline{W_{n}}$, $H_0 = \overline{W_{n}}\backslash v_{n}$, and $H_{\lfloor{\frac{n-2}{4}}\rfloor} = K_{\lfloor{\frac{3(n-1)}{4}}\rfloor + 1}$, we see $H_{k}$ is a graph obtained from $H_{k-1}$ (for $1 \leq k \leq n-1$) by applying exactly one edge contraction. We construct bounds on the number $|E({H_k})|$ of edges in each successive graph.

We have $\deg_{H_0}(v) = n-4$  for all $v \in V(H_0)$, that is $H_0$ is $(n-4)$-regular, and  $|E({H_0})| =  \frac{1}{2}\sum_{v\in V(H_0)} \deg_{H_0}(v)  = \frac{(n-1)(n-4)}{2}$. 
Moreover, for each edge $(u,v) \in E(H_0)$, the number of shared neighbors of $u$ and $v$ is $|N_{H_0}(u) \cap N_{H_0}(v)| \geq n-7$. 
This implies the \textit{minimum} number $L_0$ of edges lost by performing one edge contraction in $H_0$ satisfies 
$L_0 = \min_{(u,v)\in E(H_0)}|N_{H_0}(u) \cap N_{H_0}(v)|+1 \geq n-6$.
We define $L_k$ the minimum number of edges lost by performing one edge contraction in $H_k$.
 If we define a variable $$\Delta E_k = |E({H_{k-1}})| - |E({H_k})|$$ for the successive differences in edge set sizes, we see $\Delta E_k \ge L_k$.  

We note that each edge contraction decreases  the number of common neighbors (between two vertices) by at most one. 
In general, for $k \in [1, \lfloor{\frac{n-2}{4}}\rfloor]$ we have that if $L_{k-1} \geq \lambda$, then $ L_k \geq \lambda -1$, with $L_0\ge n-6$.
Then  $$\Delta E_k = (n - 6) - (k-1)= (n-5)-k.$$
 From this, we can evaluate $$|E({H_k})| = |E({H_0})| - \sum_{i = 1}^{k} \Delta E_i\le\frac{(n-1)(n-4)}{2}-k(n-5) +\frac{k(k+1)}{2},$$  for $k \in [1, \lfloor{\frac{n-2}{4}}\rfloor]$. 

In particular, this implies $|E(H_{\lfloor{\frac{n-2}{4}}\rfloor})| \leq \frac{n^2 - 5n + 4}{2} - \frac{(2n - 11 - \lfloor{\frac{n-2}{4}}\rfloor)(\lfloor{\frac{n-2}{4}}\rfloor)}{2}$, where $H_{\lfloor{\frac{n-2}{4}}\rfloor}$ is the graph resulting from the $\lfloor{\frac{n-2}{4}}\rfloor^{th}$ operation on $H_0$ and $|V(H_{\lfloor{\frac{n-2}{4}}\rfloor})| = \lfloor{\frac{3(n-1)}{4}}\rfloor + 1$. 
However, we have $|E(K_{\lfloor{\frac{3(n-1)}{4}}\rfloor + 1})| = \binom{\lfloor{\frac{3(n-1)}{4}}\rfloor + 1}{2}.$ 
By the calculations done in Table \ref{table1}, addressing all $n \pmod 4$, we see that   $$|E(H_{\lfloor{\frac{n-2}{4}}\rfloor})| < |E(K_{\lfloor{\frac{3(n-1)}{4}}\rfloor + 1})|.$$
This implies $ H_{\lfloor{\frac{n-2}{4}}\rfloor} \subsetneq K_{\lfloor{\frac{3(n-1)}{4}}\rfloor + 1},$  and we conclude  that  $K_{\lfloor{\frac{3(n-1)}{4}}\rfloor + 1} \not \mleq \overline{W_{n}}$.

\begin{table}[h]
\begin{tabular}{|c|c|c|c|c|}
\hline
$n $  & $\lfloor{\frac{n-2}{4}}\rfloor$ & $\lfloor{\frac{3(n-1)}{4}}\rfloor$ & $|E(H_{\lfloor{\frac{n-2}{4}}\rfloor})| <$ & $|E(K_{\lfloor{\frac{3(n-1)}{4}}\rfloor + 1})|$ \\ \hline
$4s$   & $s-1$                           & $3s-1$                             & $\frac{9s^2 - 3s - 6}{2}$                   & $\frac{9s^2 - 3s}{2}$                           \\ \hline
$4s+1$ & $s-1$                           & $3s$                               & $\frac{9s^2 + 3s - 8}{2}$                     & $\frac{9s^2 + 3s}{2}$                           \\ \hline
$4s+2$ & $s$                             & $3s$                               & $\frac{9s^2 + 3s - 2}{2}$                     & $\frac{9s^2 + 3s}{2}$                           \\ \hline
$4s+3$ & $s$                             & $3s+1$                             & $\frac{9s^2 + 9s - 2}{2}$                     & $\frac{9s^2 + 9s + 2}{2}$                       \\ \hline
\end{tabular}
\vspace{.1in}
\caption{\small Calculations show that  $|E(H_{\lfloor{\frac{n-2}{4}}\rfloor})| < |E(K_{\lfloor{\frac{3(n-1)}{4}}\rfloor + 1})|$ for all $n$ in the range.}
\label{table1}
\end{table}
\end{proof}

The \textit{Hadwiger number} of a graph, introduced by Hadwiger in 1943 \cite{H}, is defined to be the order of the largest complete minor of the graph. 
The following corollary of Theorems \ref{theo:wheel} and \ref{sharpwheelK} gives the Hadwiger number of complements of wheel graphs.

\begin{corollary} For $n\ge 6$, the Hadwiger number of $\overline{W_n}$ is $\lfloor{\frac{3(n-1)}{4}}\rfloor $.
\label{Hadwigerwheel}
\end{corollary}

\section{Elongated Triangular Prisms}

An elongated triangular prism graph  $G$ is  constructed from the triangular prism graph  in Figure \ref{prism}(a) by subdividing any or all of the edges $(u_A,w_A)$, $(u_B,w_B)$, and $(u_C,w_C)$. 
Note that in $G$, there exists exactly one path $P_i$ from $u_i$ to $w_i$ that does not contain an edge in either triangle  $u_Au_Bu_C$ or $w_Aw_Bw_C$, for $i\in \{A, B, C\}$.
 Moreover, the path $P_i$ is an induced subgraph of $G$.

%
%
\begin{lemma}
\label{prismcase}
For a graph $G$ of order 11, if $G$ is an elongated triangular prism, then $K_7 \mleq \overline{G}$.
\end{lemma}

\begin{proof}
The graph  $G$ is obtained by subdividing the non-triangular edges of the prism graph five times. 
The 5 vertices are contained in the path $P_i$ for exactly one $i \in \{A, B, C\}$. 
In this way, the distinct graphs $G$ up to isomorphism correspond to partitions of five vertices into three indistinguishable, disjoint sets. This can be done in five ways, so there are five elongated triangular prism graphs $G$ of order 11 up to isomorphism. We name these $G_{5,0,0}, G_{4,1,0}, G_{3,2,0}, G_{3,1,1}, G_{2,2,1}$ where the subscript corresponds to the partition which gives rise to the respective graph. See Figure \ref{5prisms}.

\begin{figure}[htpb!]
\begin{center}
\begin{picture}(420, 285)
\put(0,0){\includegraphics[width=6in]{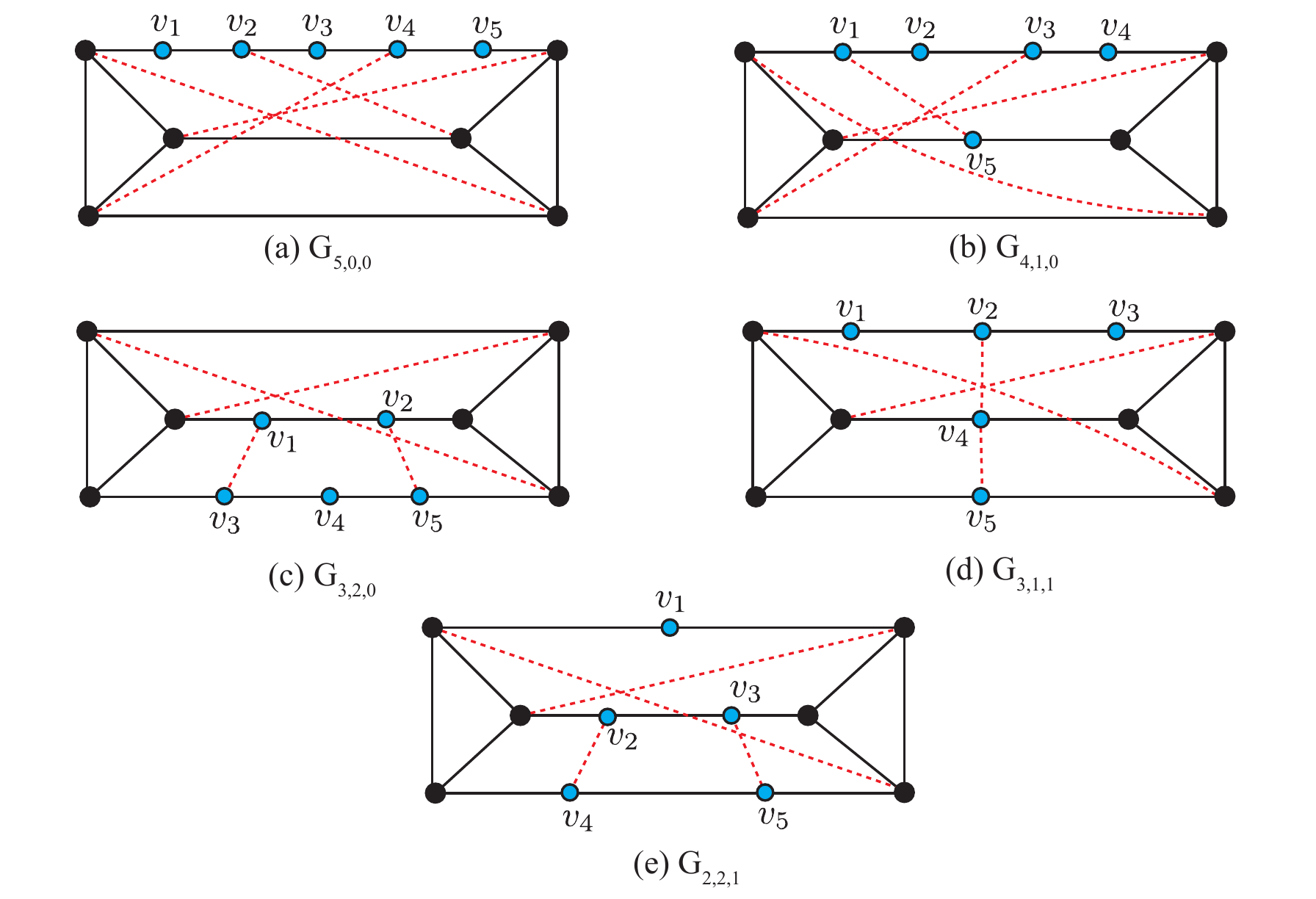}}
\end{picture}
\caption{\small The five elongated triangular prisms of order 11 up to isomorphism. The dotted edges belong to the complement graph.}
\label{5prisms}
\end{center}
\end{figure} 

A  sequence of four edge contractions performed on $\overline{G}$, as specified by the red, dotted edges in the figures, produces $K_7$ minors of $\overline{G}$ in all cases. 
\end{proof}

\begin{theorem}
\label{triangular_prism_general}  
Every elongated triangular prism graph $G$ on $|V(G)| = 2n-3$ vertices satisfies $K_n \mleq \overline{G}$ for $n \geq 7$.
\end{theorem}
\begin{proof}
Using Lemma \ref{prismcase} as our base case, we proceed by induction on $n$. 
We assume any elongated triangular prism graph $G_0$ with $|V(G_0)| = 2k-3$ satisfies $K_k \mleq \overline{G_0}$ and show an arbitrary elongated triangular prism graph $G$ with $|V(G)| = 2k-1$ satisfies $K_{k+1} \mleq \overline{G}$. 
There are two cases to consider: either all subdivisions in $G$ of the triangular prism graph occur on one path $P_i$, or there are at least two paths in $G$ with subdivisions, namely we have $|V(P_i)|, |V(P_j)| \geq 3$ for some distinct $i,j \in \{A, B, C\}$.

To consider the first case, let us say without loss of generality $|V(P_A)| = 2k - 5$ so that $(u_B, w_B), (u_C, w_C) \in E(G)$. 
We label $V(P_A) = \{u_A, v_1, v_2, \dots, v_{2k-7}, w_A\}$ such that the vertices $v_1, v_2, \dots, v_{2k-7}$ occur in this order along $P_A$ from $u_A$ to $w_A$. 
Since $k \geq 7$, we know $(v_1, v_{2k-7}) \not \in E(G)$, so $(v_1, v_{2k-7}) \in E(\overline{G})$. Let $H \mleq \overline{G}$ be the minor obtained from $\overline{G}$ by contracting the edge $(v_1, v_{2k-7})$, and call the new vertex formed by this edge contraction $v_0$. Because $N_G(v_1) \cap N_G(v_{2k-7}) = \emptyset$, we have $N_{H}(v_0) = V(H) \backslash \{v_0\}$. Since $v_0$ is adjacent to every other vertex in $H$, we now consider $H\backslash v_0$. The complement $\overline{H\backslash v_0}$ of the graph is a spanning subgraph of some elongated prism graph $G_0$ with $|V(G_0)| = 2k-3$. The assumption $K_k \mleq \overline{G_0}$ implies $K_k \mleq \overline{\left(\overline{H\backslash v_0}\right)} \cong H\backslash v_0$. It follows $$\overline{G} \mgeq H \mgeq K_1 + K_k \cong K_{k+1}.$$

For the second case, assume without loss of generality $|V(P_A)|, |V(P_B)| \geq 3$. Let $v_a \in V(P_A)$ where $v_a \not \in \{u_A, w_A\}$ and $v_b \in V(P_B)$ where $v_b \not \in \{u_B, w_B\}$. 
We have  $ (v_a, v_b) \in E(\overline{G})$ and $N_G(v_a) \cap N_G(v_b) = \emptyset$.
 As before, let $H \mleq \overline{G}$ be the minor obtained from $\overline{G}$ by contracting $(v_a, v_b)$, and call the resulting vertex $v_0$. 
Then $N_{H}(v_0) = V(H) \backslash \{v_0\}$ and the complement $\overline{H\backslash v_0}$ of the graph is a spanning subgraph of some elongated prism graph $G_0$ with $|V(G_0)| = 2k-3$. 
As before,  $K_k \mleq \overline{G_0}$ implies $K_k \mleq  H\backslash v_0$ and  it follows that $$\overline{G} \mgeq H \mgeq K_1 + K_k \cong K_{k+1}.$$ 

In both cases, the inductive step follows and completes the proof.
\end{proof}

\section{Outerplanar Graphs}

For a maximal outerplanar graph $G$ on $n$ vertices consider the planar embedding in $\R^2$ which consists of a triangulated convex regular $n$-gon $C_{n} \subseteq G$.
See  Figure \ref{outplanar10}(a).
The graph $G$ is uniquely determined up to isomorphism by this embedding, and we do not distinguish between the two.
Label the vertices of $G$ by  $v_1, v_2, \dots, v_n$,  in the order they  appear in $C_n$.  
The endpoints of any edge of $G$ which is not an edge of $C_n$ determine two paths along $C_n$. 
If the length of the shorter of the two paths is $i$, the edge is an $i$-\textit{chord} in $G$.

\begin{lemma}
\label{outerplanarcase}
For $G$  a maximal outerplanar graph of order 11,  $K_7 \mleq \overline{G}$.
\end{lemma}

\begin{proof} Consider $G$ an outerplanar graph labeled as above.
The chords of $G$ can be $i$-chords, for $i=2,3,4,5$.
 If $G$ has no 5-chord, up to a rotation, the graph $G$ contains a subgraph isomorphic to the graph $H$ pictured in Figure \ref{maxoutplanar}(a). 
Contracting the edges $(v_1,v_7)$, $(v_2,v_8)$, $(v_3,v_9)$, and $(v_5,v_{11})$ in $\overline{G}$ creates a $K_7$ minor.

\begin{figure}[htpb!]
\begin{center}
\begin{picture}(400, 120)
\put(-50,0){\includegraphics[width=7in]{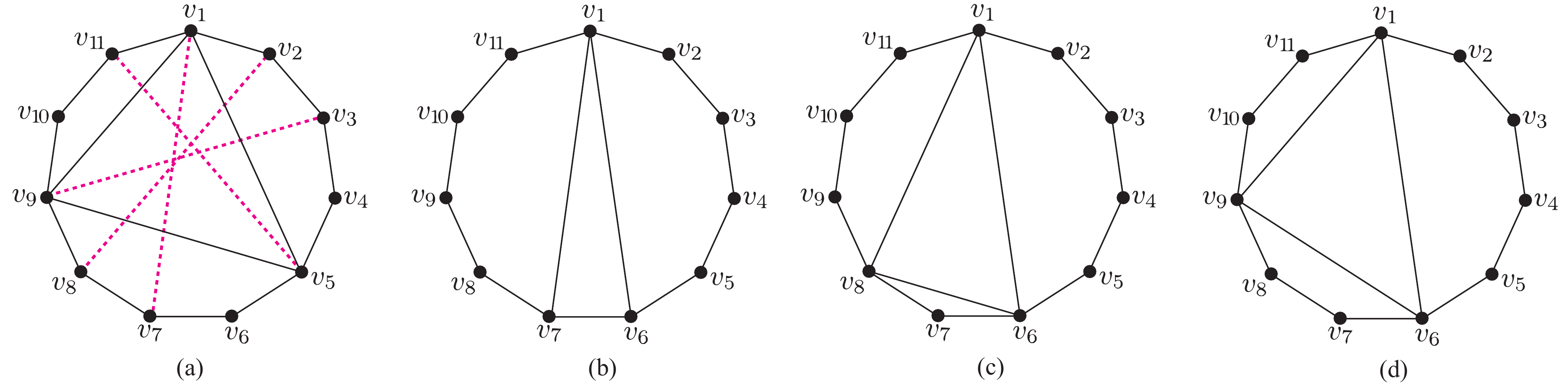}}
\end{picture}
\caption{\small (a)  $H$ is a  subgraph of $G$. Dotted edges are part of the complement $\overline{G}$ (b) Case 1: $(v_1,v_7)\in E(G)$. (c) Case 2:  $(v_1,v_8), (v_6,v_8)\in E(G)$.   (d) Case 3:   $(v_1,v_9), (v_6,v_9)\in E(G).$}
\label{maxoutplanar}
\end{center}
\end{figure} 

If $G$ has a 5-chord, say $(v_1,v_6)$, then the complement graph $\overline{G}$ contains as a subgraph a complete bipartite graph $K_{4,5}$ with vertex partitions $A=\{v_2, v_3, v_4, v_5\}$ and $B=\{v_7, v_8, v_9, v_{10}, v_{11}\}$.
We observe that at least  one of the edges $(v_2,v_4)$ and $(v_3,v_5)$ belongs to $\overline{G}$. 
Assume  $(v_2,v_4)\in E(\overline{G})$. 
In $\overline{G}$, contracting edges $(v_3,v_k)$ and $(v_5,v_j)$, where $k,j\in \{7, 8, 9, 10, 11\}, k\ne j$, and deleting the vertices $v_1$ and $v_6$  yields a minor isomorphic to $K_7$  minus a triangle rooted in the vertices in $B$.
The missing triangle is $v_pv_qv_r$ with $\{p, q, r\}= \{7, 8, 9, 10, 11\}\setminus\{k,j\}$.
To show that $\overline{G}$ contains a $K_7$ minor, it suffices to show that within the subgraph induced by $\{v_1, v_6, v_7, v_8, v_9, v_{10}, v_{11}\}$ there either exists a triangle with vertex set in $B$ or edges incident to $v_1$ or $v_6$ can be contracted to give such a triangle. 
The chord $(v_1,v_6)$ is part of a triangle with the third vertex in $B$. Up to symmetry, we distinguish three cases.\\
Case 1: $(v_1,v_7)\in E(G)$. See  Figure \ref{maxoutplanar}(b). Then $(v_6,v_8), (v_6,v_9), (v_6,v_{10}), (v_6,v_{11})\in E(\overline{G})$. At least  one of the two edges $(v_9,v_{11})$ and $(v_8,v_{10})$ belongs to $E(\overline{G})$. 
  If $(v_8,v_{10}) \in E(\overline{G})$, contract $(v_6,v_9)$ to obtain the triangle $v_8v_9v_{10}$.
If $(v_9,v_{11}) \in E(\overline{G})$, contract $(v_6,v_{10})$ to obtain the triangle $v_6v_9v_{10}$.\\
Case 2:  $(v_1,v_8), (v_6,v_8)\in E(G)$.  See  Figure \ref{maxoutplanar}(c). Then $(v_6,v_9), (v_6,v_{10}), (v_7,v_9), (v_7,v_{10})\in E(\overline{G})$. 
Contract the edge $(v_6,v_{10})$ to obtain the triangle $v_7v_9v_{10}$.\\
Case 3:   $(v_1,v_9), (v_6,v_9)\in E(G)$. See  Figure \ref{maxoutplanar}(d). Then  $(v_6,v_{10}), (v_6,v_{11}), (v_7,v_{10}), (v_7,v_{11})\in E(\overline{G})$. 
Contract the edge $(v_6,v_{10})$ to obtain the triangle $v_7v_{10}v_{11}$.\\

\end{proof}

To tackle the general case, we consider the relative position of chords in  maximal outerplanar graphs in Lemma \ref{classmaxouter}. 
We say two edges $e_1, e_2 \in E(G)$ are \textit{independent} if and only if $V(e_1) \cap V(e_2) = \emptyset$. 
We start by remarking that ever maximal outerplanar graph with $n\ge 4$ has at least one 2-chord.

\begin{lemma}
\label{classmaxouter}
For a maximal outerplanar graph $G$ of order $n \geq 7$, exactly one of the following is true:
\begin{enumerate}
    \item \label{itm:twoind} there exist a pair of \textit{independent} edges $e_1, e_2 \in E(G)$ such that both $e_1$ and $e_2$ are 2-chords of $G$.
    \item \label{itm:pickle} $G$ is isomorphic to $K_1 + P_{n-1}$
\end{enumerate}
\label{indedges}
\end{lemma}
\begin{proof}

Let $v_1, v_2, \ldots , v_n$ be the vertices of $G$ in the order in which they appear in $C_n$, the boundary of the common face.
Let $k$ be the maximal length of all the chords of $G$, and let $(v_i,v_{i+k})$ be a $k$-chord of $G$.
(a) If there exist a chord $(v_j,v_{j+l})$ independent of $(v_i,v_{i+k})$, we may assume the vertices $v_i, v_{i+k}, v_j, v_{j+l}$ appear in this order along $C_n$. See Figure \ref{fig-chords}(a).
The subgraph $H_1$ of $G$ induced by $\{v_i, v_{i+1}, \ldots, v_{i+k}\}$ has a 2-chord.
If this edge is a 2-chord of $G$, call it $e_1$. 
Else, we may assume this chord is $(v_i,v_{i+k-1})$.
The subgraph  $H_2$ induced by $\{v_i, v_{i+1}, \ldots, v_{i+k-1}\}$ has a 2-chord.
If this edge is a 2-chord of $G$, call it $e_1$. 
Else, consider $H_3$ the subgraph obtained from $H_2$ by deleting the vertex $v\in \{v_i, v_{i+k-1}\}$, which is not an endpoint of the 2-chord. 
Continue the process until a 2-chord $e_1$ of $G$ is found.
The latest this chord could be found is in  a subgraph induced by four consecutive vertices of $H_1$.
Similarly, find a 2-chord in the subgraph of $G$ induced by $\{v_j, v_{j+1}, \ldots, v_{j+l}\}$ and call it $e_2$. 
Then, $e_1$ and $e_2$ are independent 2-chords of $G$. \\
(b) If every chord of $G$ shares an endpoint with $(v_i,v_{i+k})$,
let $H_1$ be a subgraph of $G$ induced by $\{v_i, v_{i+1}, \ldots, v_{i+k}\}$. 
Without loss of generality, we may assume that there exists $t$ such that $(v_i,v_t)\in E(H_1)$. 
See Figure \ref{fig-chords}(b).
Let $H_2$ be the subgraph of $G$ induced by $V(G) \setminus\{v_{i+1}, v_{i+2}, \ldots, v_{i+k-1}\}$. 
If there exists a chord $(v_{i+k},v_s)\in E(H_2)$, then $(v_i,v_t)$ and $(v_{i+k},v_s)$ are independent chords. 
By part (a), $G$ has two independent 2-chords.
If such a chord does not exist, all chords of $H_2$ meet at $v_i$.
If all chords of $H_1$ meet at $v_i$, then $G\cong K_1+P_{n-1}$.
Else, a chord of $H_2$ incident to $v_i$ is independent of a chord of $H_1$ incident to $v_{i+k}$ and by part (a), $G$ has two independent 2-chords.

\begin{figure}[htpb!]
\begin{center}
\begin{picture}(200, 125)
\put(-50,0){\includegraphics[width=4in]{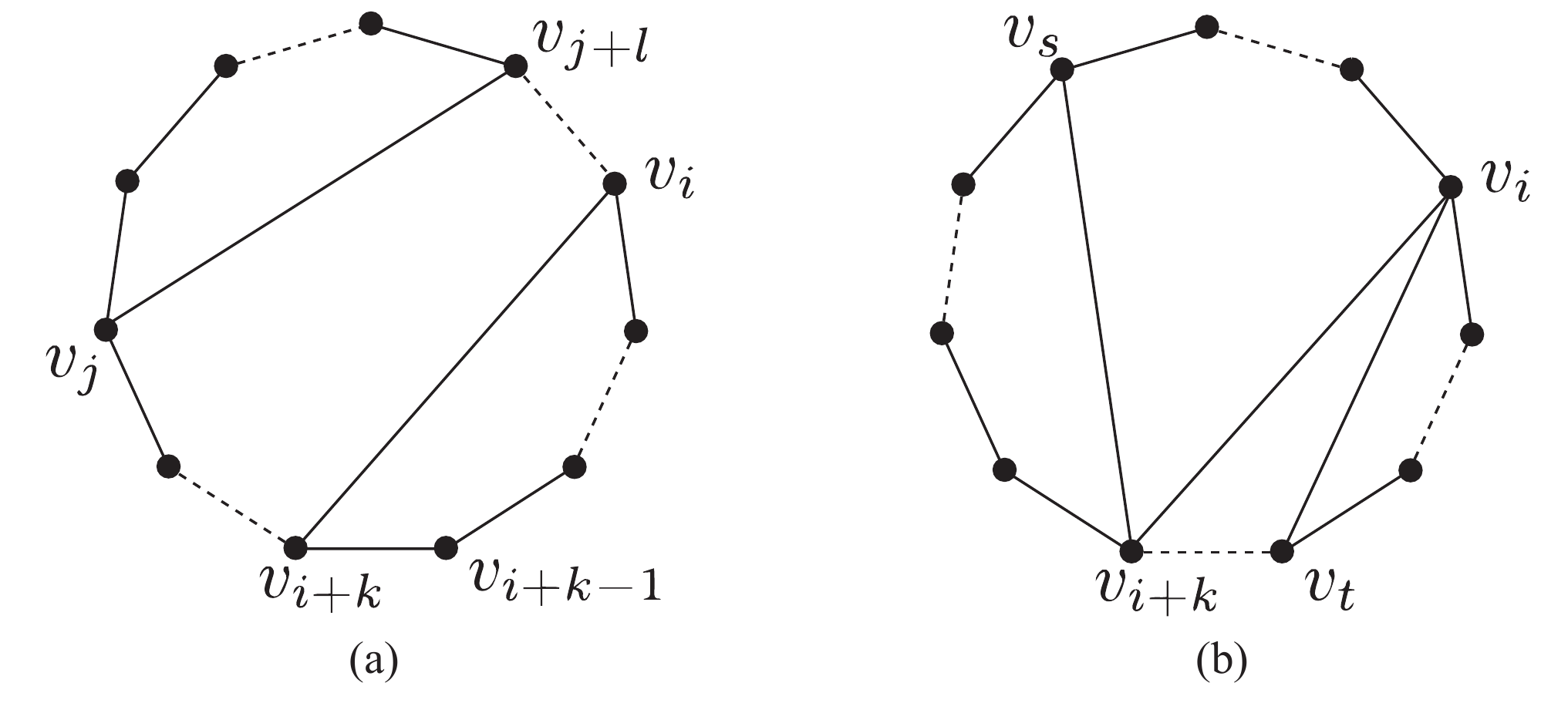}}
\end{picture}
\caption{\small (a) $G$ has independent chords $(v_i,v_{i+k})$ and $(v_j,v_{j+l})$; (b)  all the chords of $G$ share a vertex with $(v_i,v_{i+k})$.}
\label{fig-chords}
\end{center}
\end{figure}

\end{proof}

\begin{theorem}
\label{maxoutergeneral}
Every maximal outerplanar graph $G$ of order $|V(G)| = 2n-3$, $n\geq 7$,  satisfies $K_n \mleq \overline{G}$.
\end{theorem}

\begin{proof}
We use induction on the number of vertices $n\ge 7$. The base case, $n=7$ is provided by Lemma \ref{outerplanarcase}.
Let $G$ denote an outerplanar graph with $2(n+1)-3 = 2n-1$ vertices. 
By Lemma \ref{indedges}, we need to consider two cases.
\begin{enumerate}
\item There exists a pair of independent 2-chords $(v_{i-1}, v_{i+1})$ and $(v_{j-1}, v_{j+1})$.
Then $N_G(v_i)\cap N_G(v_j)=\emptyset$ and $H:= G-\{v_i, v_j\}$ is a maximal outerplanar graph  of order $2n-3$.
In turn  $N_{\overline{G}}(v_i) \cup N_{\overline{G}}(v_j) = V(G)$.
By the induction hypothesis, $ \overline{H}$ contains a complete minor of order $n$.
Contracting the edge $(v_i, v_j)$ in $\overline{G}$ yields a minor isomorphic to $H+K_1$ which has a complete minor of order $n+1$.
\item If $G$ is isomorphic to $K_1 + P_{2n-4}$,  then $G$ is a spanning subgraph of $W_{2n-3}$. 
By Corollary \ref{cor}, $K_n \mleq \overline{W_{2n-3}}$, and thus $K_n \mleq \overline{G}$.

\end{enumerate}

\end{proof}

Complements of maximal outerplanar graphs with $2n-4$ vertices do not necessarily have  a $K_n$ minor, as the following example illustrates.

\begin{example}
There exists a maximal outerplanar  graph $M$ of order $10$ for which $K_7 \not \mleq \overline{M}$.
\label{examplemaxoutpla10}
\end{example}

\begin{figure}[htpb!]
\begin{center}
\begin{picture}(330, 120)
\put(0,0){\includegraphics[width=5in]{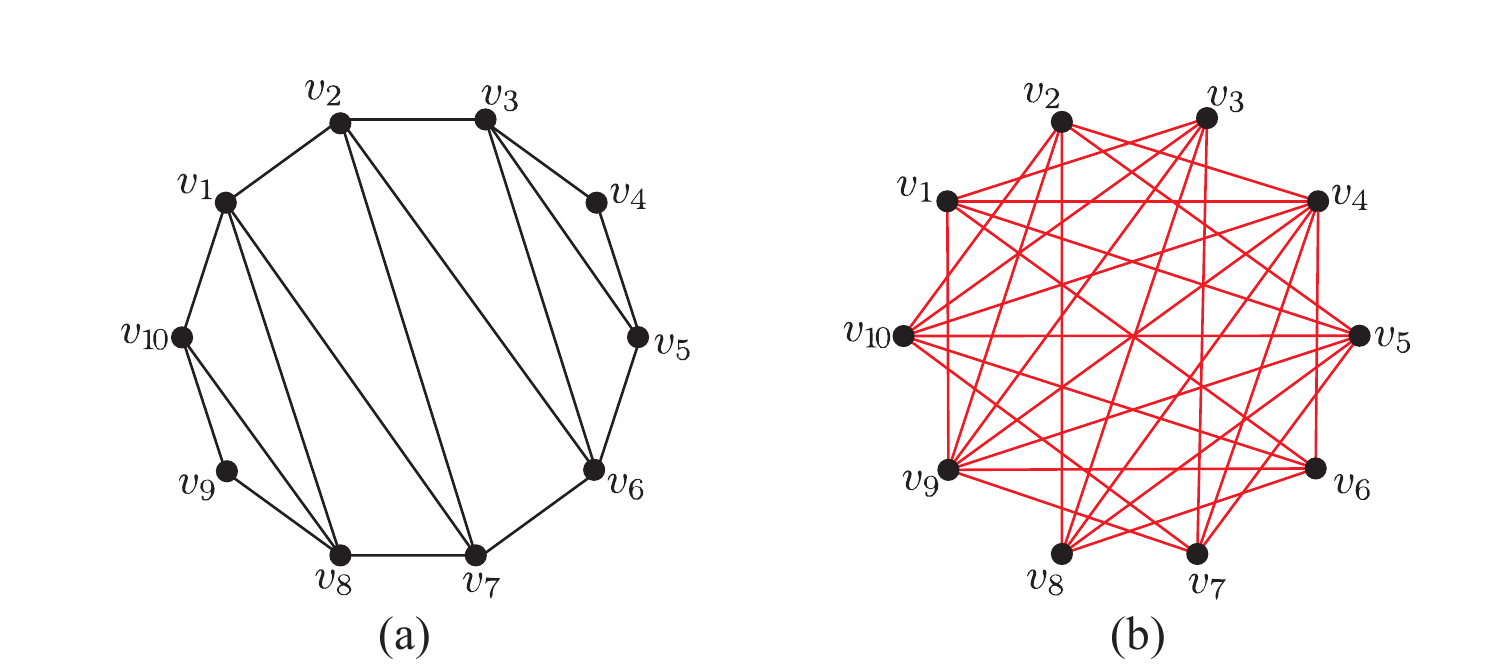}}
\end{picture}
\caption{ \small (a) Graph $M$ in Example \ref{examplemaxoutpla10} ; (b) Graph $\overline{M}$ in Example \ref{examplemaxoutpla10}.}
\label{outplanar10}
\end{center}
\end{figure} 

\begin{proof}
Consider the graph $M$ and its complement $\overline{M}$ as shown and labeled in Figure \ref{outplanar10}. 
We see that in  $\overline{M}$ the vertices  $v_1, v_2, v_3, v_6, v_7, v_8$ all have degree at most 5. 
This implies $K_7 \not \subseteq \overline{M}$. 
If $\overline{M}$ contains a $K_7$ minor, this  minor is  obtained from $\overline{M}$ by  three edge contractions. 
Under an edge contraction, the degree of a vertex not incident to the  contracted  edge is non-increasing. 
The number of vertices of degree at least six  increases by at most 1 after each edge contraction, and this occurs only if (but not necessarily) the endpoints of the contracted edge both have degree less than six. 
The vertices $v_1, v_2, v_3, v_6, v_7, v_8$ of $G$ all have degree less than six.
To achieve a $K_7$ minor, all endpoints of the three contracted edges must belong to this set, and the only possibility is contracting the edges $(v_1, v_6)$, $(v_2, v_8),$ and   $(v_3, v_7)$.
By inspection, the  graph obtained by contracting these three edges is not complete and so  $\overline{M}\not \mgeq K_{7}$.
\end{proof}


\section{Future Explorations}
It is worth noting that the main result of this article is specific to \textit{non-separating} planar graphs. 
Consider the maximal planar graph $G$ with eleven vertices in Figure \ref{planarcomplement}(a) and its complement  $\overline{G}$ in Figure \ref{planarcomplement}(b). 
The graph $\overline{G}$ has edge set $\{(v_1,v_2),(v_1,v_3),(v_1,v_4),(v_1,v_5),(v_1,v_6), \\(v_1,v_7),( 
v_1 ,v_8), ( v_2 , v_4), ( v_2 , v_7),( 
 v_2 , v_{11}),( v_3 , v_6),( 
v_3 , v_{10}),$ $( v_3 , v_{11}),( 
 v_4 , v_6),( v_4 , v_7),( v_4 , v_8),\\( 
 v_5 , v_8),( v_5 , v_9), ( v_5 , v_{11}),(
  v_6 , v_9),( v_6 , v_{10}),( 
 v_6 , v_{11}),( v_7 , v_9),( 
 v_7 , v_{10}),( v_7 , v_{11}),( 
 v_8 , v_9),( v_8 , v_{10}),\\( 
 v_8 , v_{11})\}$.
 We show  $\overline{G}$ does not have a $K_7$-minor. 
Since $\deg_{\overline{G}}(v_2)=4$, to obtain a $K_7$ minor of $\overline{G}$ an edge incident to the vertex $v_2$ needs to be contracted.
Since the edges $(v_1,v_4)$, $(v_4,v_7)$ and $(v_1,v_7)$ are contained in the neighborhood of $v_2$, if follows that $\overline{G}$ has a $K_7$ minor if and only if the graph $\overline{H}$ obtained by contracting the edge $(v_2,v_{11})$ of $\overline{G}$  has a $K_7$ minor. See Figure \ref{10contraction}.

\begin{figure}[htpb!]
\begin{center}
\begin{picture}(420, 150)
\put(0,0){\includegraphics[width=6in]{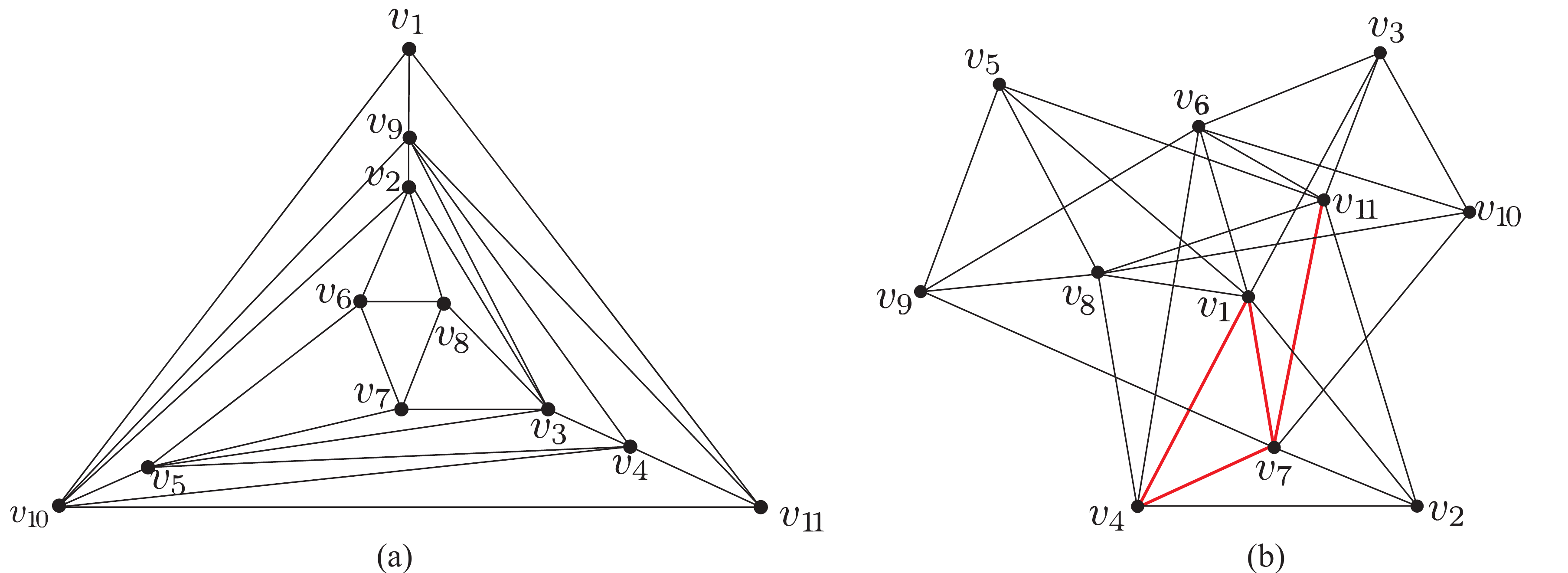}}
\end{picture}
\caption{\small (a) A maximal planar graph $G$ of order 11 (b) The graph $\overline{G}$, the complement of a maximal planar graph of order 11.}
\label{planarcomplement}
\end{center}
\end{figure}

 \begin{figure}[htpb!]
\begin{center}
\begin{picture}(350, 180)
\put(0,0){\includegraphics[width=4.5in]{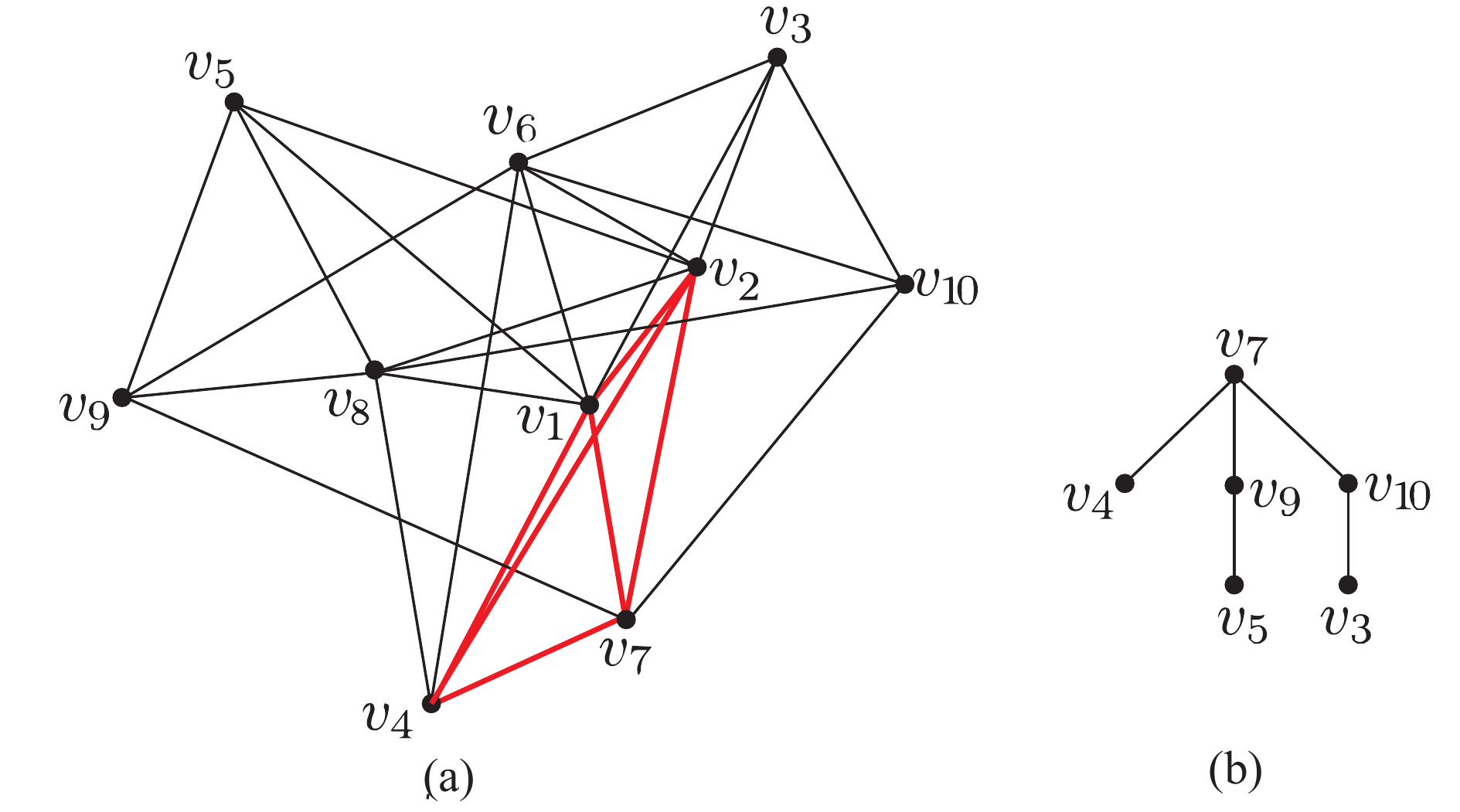}}
\end{picture}
\caption{(a) The graph $\overline{H}$  obtained by contracting the edge $(v_2,v_{11})$ of $\overline{G}$ (b) The graph induced in $\overline{H}$ by $\{v_3, v_4, v_5, v_7, v_9 , v_{10}\}$. }
\label{10contraction}
\end{center}
\end{figure} 
 Note that in $\overline{H}$, the vertices $v_3, v_4, v_5, v_7, v_9$, $v_{10}$ each have degree less than 6. 
 If a $K_7$ minor of $\overline{H}$ exists, it is obtained by contracting three edges of $\overline{H}$.
 Vertices $v_3, v_4, v_5, v_7, v_9$, $v_{10}$ all must be incident to the contracted edges.
  However, since the subgraph of $\overline{H}$ induced by $\{v_3, v_4, v_5, v_7, v_9 , v_{10}\}$ is a tree with leaves $v_3, v_4, v_5$,  it must be that the edges to be contracted are $(v_3, v_{10})$, $(v_4,v_7)$, and $(v_5,v_9)$. See Figure \ref{10contraction}(b).
   Contracting these edges, however, produces a minor of order 7 and size 19, thus not isomorphic to $K_7$.
 
  On the other hand, contracting the edges $(v_2,v_{11})$, $(v_3,v_{10})$, and $(v_5,v_9)$ in $G$, yields a minor isomorphic to $K_{3,3,1,1}$, which is a minor minimal intrinsically knotted graph, by the work of Foisy \cite{Foisy}.  It follows that the graph $G$ is intrinsically knotted. 
 
 More can be said about the complements of planar graphs of order 11: they are all intrinsically knotted. Independently, Foisy \cite{Foisy} and Taniyama and Yasuhara \cite{TY} provided a sufficient condition for a graph to be intrinsically knotted. Namely, if every embedding of the graph contains a double-linked $D_4$-minor (a graph of order 4 with a set of 4 double edges, which form a set of two pairs of linked cycles), then the graph is intrinsically linked. Based on these results, Miller and Naimi \cite{MN} developed and algorithm which checks whether a given graph has a double-linked $D_4$-minor in every embedding. Naimi implemented the algorithm into a Mathematica program. We used this program to confirm that each of the complements of the 1249 maximal planar graphs of order 11 is intrinsically knotted. Thus:

\begin{proposition} The complement of a planar graph of order 11 is intrinsically knotted. 
\label{conjecture1}
\end{proposition}

 From a topological perspective, Proposition \ref{conjecture1} is the natural continuation of the work of Battle, Harray, and Kodama \cite{BHK}, and Tutte \cite{Tutte}, who proved that the complement of a planar graph of order 9 is not planar, and the work of Lov\'asz and Shrijver \cite{LS}, and Kotlov, Lov\'asz and Vempala \cite{KLV}, which implies that the complement of a planar graph of order 10 is intrinsically linked. As always, a proof of Proposition \ref{conjecture1} without computer assistance might provide extra insight into the structure of the complements of planar graphs.

Theorem \ref{mainIK} shows the complement of a non-separating planar graph of order $2n-3$ contains a $K_n$ minor. It would be interesting to investigate, using techniques similar to those in this article, the connection between the order of a (non-separating) planar graph and the existence of a $K_{3,3}+K_t$ minor for the complement of the graph.

\section{Acknowledgements} The authors would like to thank Elena Pavelescu for the useful comments and suggestions.

\end{document}